\documentclass[article]{amsart}
\usepackage{bbm}
\usepackage{amsfonts}
\usepackage{mathrsfs}
\usepackage{hyperref}
\usepackage{amssymb}
\usepackage{CJK}
 \newtheorem{thm}{Theorem}[section]
 \newtheorem{cor}[thm]{Corollary}
 \newtheorem{prop}[thm]{Proposition}
 
 \newtheorem{conj}[thm]{Conjecture}
 \newtheorem{lem}[thm]{Lemma}
 
  \newtheorem{prob}[thm]{Problem}
 \newtheorem{defn}[thm]{Definition}
 \theoremstyle{remark}
 
\newcommand{\f}{\mathbb{F}_q}
\newcommand{\fm}{\mathbb{F}_{q^m}}
\newcommand{\ft}{\mathbb{F}_{q^2}}
\newcommand{\fh}{\mathbb{F}_{q^h}}
\newcommand{\fk}{\mathbb{F}_{q^k}}
\newcommand{\ff}{\mathbb{F}_p}
\newcommand{\fp}{\mathbb{F}_{p^p}}

\begin{document}
\title[On counting subset sums over finite abelian groups]
  { On the construction of small subsets containing special elements in a finite field}

\author{Jiyou Li}
\address{Department of Mathematics, Shanghai Jiao Tong University, Shanghai, P.R. China}
\email{lijiyou@sjtu.edu.cn}

\thanks{This work is supported by the National Science Foundation of China (1771280) and the National Science Foundation of Shanghai Municipal (17ZR1415400). }





\begin{abstract}
In this note we construct a series of small subsets
containing a non-d-th power element in a finite field
by applying certain bounds on incomplete character sums.

 Precisely, let $h=\lfloor q^{\delta}\rfloor>1$ and $d\mid q^h-1$.
   Let $r$ be a prime divisor of $q-1$ such that the largest prime power part of $q-1$ has  the form $r^s$.
   Then there is a constant $0<\epsilon<1$ such that for a ratio at least $ {q^{-\epsilon h}}$ of $\alpha\in \mathbb{F}_{q^{h}} \backslash\mathbb{F}_{q}$,  the set $S=\{ \alpha-x^t, x\in\mathbb{F}_{q}\}$
of cardinality $1+\frac {q-1} {M(h)}$
  contains a non-d-th  power in $\mathbb{F}_{q^{\lfloor q^\delta\rfloor}}$, where $t$ is the largest power of $r$ such that $t<\sqrt{q}/h$ and $M(h)$ is defined as
   $$M(h)=\max_{r \mid (q-1)} r^{\min\{v_r(q-1),  \lfloor\log_r{q}/2-\log_r h\rfloor\}}.$$
    Here $r$ runs thourgh prime divisors and $v_r(x)$ is the $r$-adic oder of $x$.
For odd $q$, the choice of $\delta=\frac 12-d, d=o(1)>0$ shows
that there exists an explicit subset of cardinality $q^{1-d}=O(\log^{2+\epsilon'}(q^h))$ containing a non-quadratic element in  the field $\mathbb{F}_{q^h}$.
  On the other hand, the choice of $h=2$ shows that for any odd prime power $q$, there is an explicit subset of cardinality $1+\frac {q-1}{M(2)}$ containing a non-quadratic element in $\mathbb{F}_{q^2}$. This improves a $q-1$ construction by Coulter and Kosick \cite{CK} since $\lfloor \log_2{(q-1)}\rfloor\leq M(2) < \sqrt{q}$.



  In addition, we obtain a similar construction for small sets containing a primitive element. The construction works well provided $\phi(q^h-1)$ is very small, where $\phi$ is the Euler's totient function.
\end{abstract}
\maketitle 

\section{Introduction}

For an odd prime $p$, it is a historical and hard problem to determine $q_p$,  the least quadratic non-residue or generally, the least d-th power  non-residue of $p$.
Clearly, this problem is essentially reduced to finding a nontrivial upper bound for the sum $$\sum_{1 \leq i\leq x}\chi(x),$$ where $\chi$ is a non-principal multiplicative character modulo $p$.

 There are two well-known and important upper bounds for this
quantity. The first, discovered independently by Polya and Vinogradov one century ago, asserts
$$|\sum_{1 \leq i\leq x}\chi(x)|\ll \sqrt{p}\log{p}$$
and this implies there is a constant $C$ such that
$$q_p\leq C\sqrt{p}\log{p}.$$

It is believed that the Polya-Vinogradov bound can be improved to
$$|\sum_{1 \leq i\leq x}\chi(x)|\ll \sqrt{p}\log{\log{p}},$$
which was proved by Montgomery and Vaugham under Generalized Riemann Hypothesis(GRH).

The second,  the work of Burgess and Hildebrand (1962) [12] extended the range of $x$. It was shown that for all primitive real quadratic characters $\chi \mod p $, $$|\sum_{1 \leq i\leq x}\chi(x)|=o(x)$$
provided  $x>p^{1/4+\epsilon}$. This implies  that for every $\epsilon > 0$  there is a constant $C$ such that $${\displaystyle q_{p}\leq Cp^{{\frac {1}{4}}+\epsilon }.} $$

Note that the index $\frac 1 4$ can be replaced by $\frac 1 {4 \sqrt{e}}$ by a sieving argument for quadratic character.

It is also believed that Burgess bound can be improved to
$$|\sum_{1 \leq i\leq x}\chi(x)|=o(x)$$
provided  $x>p^{\epsilon}$, which was also known to be true under GRH. Actually  Linnik showed that under GRH, for any $\epsilon>0$, $$q_g=o(p^{\epsilon}).$$

   The best bound assuming GRH  was given by Ankeny \cite{A}:    $$q_p=O(\log^2p).$$
   In the same paper the author claimed the same bound for $d$-th power  non-residue and the least prime $g$ which is quadratic residue mod p.

    On the other hand, Chowla and Turan showed that the bound $q_p=O(\log^2p)$ is not far from the best possible, namely, there is a positive constant $C$ such that for infinitely many primes $p$ one has $$q_p > C \log p.$$

    Few explicit bounds were known.  Since an upper bound for primitive roots always works for quadratic non-residues, Grosswald's work \cite{G} shows that if ${\displaystyle p>e^{e^{24}}\approx 10^{10^{10}}},$ then ${\displaystyle q_{p}<p^{0.499}}.$
And the work of Cohen, Oliveria e Silva and Trudgian \cite{COT} shows that if  $p> 4\cdot 10^{71}$, then
  $ q_p < p^{1/2} -2 .$
A unconditional bound was given by Cohen and Trudgian \cite{CT} that for any $p$,
   $$ q_p < p^{0.96}.$$

In computational number theory and theoretical computer sciences, it is a basic problem to deterministically generate a quadratic non-residue efficiently.  For instance, generating a quadratic non-residue plays a crucial role in computing a square root of an element  mod a prime number.

The above $\log^c{p}$ type bound immediately gives a polynomial deterministic  algorithm to return a quadratic non-residue under GRH. And thus it is an interesting  problem to find a relatively small set containing desired elements.
 In fact, the deterministic output of a small subset containing special  elements such as primitive elements and quadratic non-residues was widely studied.


In this paper, thanks to the Weil's bound on function fields, we establish several constructions in finite fields by using tools from algebraic function fields. We construct a series of explicit small sets containing a quadratic non-residue, or generally a non-d-th power.

Let us start from a problem raised by Coulter and Kosick \cite{CK}:
\begin{prob}[Coulter and Kosick]
Let $A$ be a subset of $\f^*$ of order $O(q)$. What is the minimum order of a subset $S$ of $\f^*$ so that for all $a\in A$, the set $\{a-s, s\in S\}$ contains both a square and non-square element of $\f^*$?
\end{prob}

The authors constructed such subsets $S$ of cardinality roughly equal to  $\sqrt{q}$,
by using methods from algebraic combinatorics. Precisely they proved the following result.

 \begin{thm}[Coulter and Kosick]Assume $q\geq 7$ and let $A=\alpha\in\ft-\f$.
 Then for any $\alpha\in A$, the set $S=\{\alpha-a^2, a\in\f^*\}$ contains both a square and a non-square in $\ft$.
 \end{thm}

Since finding a square is a trivial task, we may ask the following essentially equivalent question:
\begin{prob}
Let $A$ be a subset of $\f^*$ of order $O(q)$. Construct small subsets $S$ of $\f^*$ so that for all $a\in A$, the set $\{a-s, s\in S\}$ contains a non-square  in $\f^*$.
\end{prob}

When the base field  is special and the field extension is large enough, one can expect stronger constructions. For instance, in \cite{HM} Heath-Brown and Micheli proved the following result.

\begin{thm}[Heath-Brown and Micheli]
 Suppose $p$ is an odd prime and $h(x)=x^p-x-a\in \ff[x]$, $\fp$ is the degree $p$ extension of $\ff$. If $a$ is not a square in $\ff$, then for any root $\alpha$ of $h(x)$, each element in the set $\alpha+\ff$ is a non-square in $\fp$.
\end{thm}

In general, they proved the following theorem, which was used to construct large family of \texttt{"}dynamically irreducible\texttt{"}  quadratic polynomials.

\begin{thm}[Heath-Brown and Micheli]
Suppose $p$ is an odd prime with $p\equiv 1 (\bmod \ 4)$.
    If $q>p^{p \cdot \frac {\sqrt{2}} {\log p}}$, then there is an element $\alpha$ in $\f \backslash \ff$  such that all of the elements in the set $\alpha+\ff$ are non-squares in $\f$.
\end{thm}

In this paper we give a class of explicit constructions of such kind of small subsets, in which the cardinality is
much smaller than previous constructions for the same extension degree. Some constructions achieve the $\log^c q$ type bound. Precisely, we have the following general result.
\begin{thm} Let $h>1$ and $d\mid q^h-1$. For any $\alpha\in\fh\backslash\f$, let $e$ be the (multiplicative) order of $\alpha$. Suppose $h, t$ satisfy the the following conditions:

 1. $(t, \frac{q^h-1}e)=1;$

 2. Each prime factor of $t$ divides $e$;

 3. $q^h\equiv 1(\bmod \ 4)$ if $t\equiv0 (\bmod\ 4)$;

 4.  $th\leq \sqrt{q}$.

Then the set  $S=V(\alpha-x^t)=\{\alpha-x^t, x\in\f\}$  contains a non-d-th power in $\fh$.
 \end{thm}

Some conditions above can be simplified in many cases.  For instance,  if $\omega$ is primitive in $\fh$, then condition $1$ can be dropped and if $q$ is odd one can always choose $t=2^k$.  To get the best construction, we need to take a prime divisor $r$ of  $q-1$
such that the largest prime power part of $q-1$ equals $r^s$ for some s.  In particular we have the following corollaries.

\begin{cor} Let $h=\lfloor q^{\delta}\rfloor>1$ and $d\mid q^h-1$.
   Let $r$ be a prime divisor of $q-1$ such that the largest prime power part of $q-1$ has  the form $r^s$.
   Then there is a constant $0<\epsilon<1$ such that for a ratio at least $ {q^{-\epsilon h}}$ of $\alpha\in \mathbb{F}_{q^{h}} \backslash\mathbb{F}_{q}$,  the set $S=\{ \alpha-x^t, x\in\mathbb{F}_{q}\}$
of cardinality $1+\frac {q-1} {M(h)}$
  contains a non-d-th  power in $\mathbb{F}_{q^{\lfloor q^\delta\rfloor}}$, where $t$ is the largest power of $r$ such that $t<\sqrt{q}/h$ and $M(h)$ is defined as
   $$M(h)=\max_{r \mid (q-1)} r^{\min\{v_r(q-1),  \lfloor\log_r{q}/2-\log_r h\rfloor\}}.$$
    Here $r$ runs thourgh prime divisors and $v_r(x)$ is the $r$-adic oder of $x$.
 \end{cor}


 The case $\delta=\frac 14$ shows the construction is almost optimal, since $q^{3/4}=O(\log {q^h})^3$.
And for $\delta=\frac 12 -\epsilon$, on can show that there is at least one $t$ such that $|S|=q^{1-\epsilon}=O(\log {q^{q^{\frac 12 -\epsilon}}})^{2+\epsilon'}$, achieving the $O(\log^{2}p)$ type bound for the prime field $\ff$ case under GRH.
In particular we have the following corollaries considering special cases.

\begin{cor}
For any $q, 1 \leq h\leq \lfloor \sqrt{q}\rfloor, d\mid q^h-1$, the set  $S=V(\alpha-x)=\{\alpha-x, x\in\f\}$ contains a non-d-th power in $\fh$.  In particular,  there is a set of cardinality $q$ in $\mathbb{F}_{q^{\lfloor \sqrt{q}\rfloor}}$ containing a non-d-th power.
\end{cor}


For another extreme case,  we  have
    \begin{cor} Let $h=2$ and $d\mid q-1$. If both $q$ and  $\frac{q^2-1}e$ are odd, then
 there is a set $S$ of cardinality $1+\frac {q-1}{M(2)}$ in $\ft$ containing a non-d-th power, where
    $$M(2)=\max_{r \mid (q-1)} r^{\min\{v_r(q-1),  \lfloor\log_r{q}/2-\log_r 2\rfloor\}}.$$
  \end{cor}

  Clearly the above construction holds when $\alpha$ is primitive.
  Let us return to the Coulter and Kosick's question.
 Note that $|A|=O(q^2)$ and the cardinality of $S$ equals $1+\frac {q-1}{M(2)}$.  Since $\lfloor \log_2{(q-1)}\rfloor\leq M(2) < \sqrt{q}$, our construction is essentially better than Coulter and Kosick's construction, which has size $q-1$. In particular, when $M(2)\sim \sqrt{q}$, our construction
 set has cardinality $O(\sqrt{q})$ in $\ft$.

  Similarly, one may consider the constructions for small sets containing a primitive elements by using approaches from \cite{Co} .
  \begin{thm}
Let $\tau(x)$ denote the number of divisors of $x$.
Suppose $t$ satisfies the following conditions

 1. $(t, \frac{q^{n}-1}e)=1;$

 2. Each prime factor of $t$ divides $e$;

 3. $q^{n}\equiv 1(\bmod \ 4)$ if $t\equiv0 (\bmod\ 4)$;

 4.  $nt\leq \sqrt{q}$.

If $\tau(q^n-1)<\frac{\sqrt{q}}{nt-1}+1$, then the set $$S=V(\alpha-x^t)=\{ \alpha-x^t, x\in\f\}$$
of cardinality $O(q/t)$ in $\mathbb{F}_{q^n}$ contains a primitive element.

In particular, if $\tau(q^h-1)<\frac{\sqrt{q}}{h-1}+1$, then the set $$S=V(\alpha-x)=\{ \alpha-x, x\in\f\}$$
of cardinality $q$ in $\mathbb{F}_{q^h}$ contains a primitive element.
\end{thm}

   Note that the above construction performs well if $\phi(q^h-1)$ is small.

 Our main tool is a bound on incomplete character sums over finite fields.

\section{Main Results}

To prove the main result, the main tool is a key lemma on incomplete character sums over finite fields deduced from
the celebrating Weil theorem.
In this paper we always let $\f$ be a finite field of cardinality $q$ and let  $\fm$  be an degree $m$ extension field.
Let $f(x)$ be a nonconstant polynomial defined over the extension field $\fm$.  Let $\chi$ be a non-trivial
multiplicative  character defined on  $\fm$. Define the following
incomplete character sum by  $$S_d(\chi)=\sum_{a\in \f}\chi(f(a)).$$
Before stating the main lemma,  here is a toy example.
\begin{prop} Suppose $q$ is odd. Let $\f[\alpha]=\ft$. Assume $\alpha$ is not a square in $\ft$, $f(x)=x^2-\alpha\in\ft[x]$. Let $\beta$ be a simple root of $f(x)$.
 Then
 $$|\sum_{a\in \f}\chi(a^2-\alpha)|\leq 3\sqrt{q}.$$
\end{prop}

\begin{thm}\cite{W} \label{main}
Let $f(x)$ be a nonconstant polynomial defined over $\fm$. Suppose that the
 largest squarefree divisor of $f(x)$ has degree $D$.
If there is a root $\zeta$ of multiplicity $t$ of $f(x)$ such that the character  $\chi^{t}$
is non-trivial on the image set $\text{Norm}_{\fm[\zeta]/\fm}\f[\zeta]$. Here $\text{Norm}_{\fm[\zeta]/\fm}$ is
the norm map from $\fm[\zeta]$ to $\fm$.  Then we have the estimate
$$|S_d(\chi)|\leq (mD-1)\sqrt{q}.$$
In particular, if $\fm[\zeta]=\f[\zeta]$, then $\chi^{t}$
is non-trivial on $\fm$, thus  the above bound always holds.
\end{thm}
For the proof and the details, please refer to \cite{W} for a nice exploration.
It follows from the above lemma and the basic theory of finite fields extensions that
\begin{prop} \label{thm2.2}Let $g(x)$ be a polynomial of degree $D$ defined over $\f$. For an $\omega\in \fm\backslash\f$,
let $f(x, \omega)=\omega-g(x)$ be defined over $\fm$ and suppose $f(x, \omega)$ is irreducible over $\fm$.
If $d\mid q^m-1$ and $mD< \sqrt{q}+1$, then the image set
$$S=\{\omega-g(a), a\in \f\}$$ contains a non-d-th power.
\end{prop}
\begin{proof}
Let $\zeta$  be a simple root of $f(x)$. Clearly $ \f[\zeta]=\f[\zeta, \omega]=\fm[\zeta]$ and the conclusion follows from Theorem \ref{main}
since $\chi$ is nontrivial on $\fm$.
\end{proof}
 There are still two key steps. The first step is to determine if a polynomial of the above type $f(x, \omega)$ is irreducible. This problem seems in general nontrivial. Even the existence of such kind of irreducible polynomials is still open.  Precisely, Munemasa  and Nakamura \cite{MN} made the following conjecture:

 \begin{conj}[Munemasa  and Nakamura]
 Let $q$ be a prime power (or equivalently, a prime), and let $k, l$ be positive integers. Then there exists a monic irreducible polynomial $f(x)\in \fk[x]$ of degree $l$ such that $f(x)-f(0)\in \f[x]$ and $f(0)$ does not  belong to any proper subfield of $\fk$.
 \end{conj}

Some partial results were obtained. The conjecture clearly  holds when $k=1$ or $l=1$. It follows from the irreducibility of binomials (see Lemma \ref{Lemma2.5}) that the conjecture holds for specified $l$'s. For instance,  the conjecture holds for $l=2^t$ when $p$ is odd.

The second step,  the determination of the cardinality of the image of $g(x)$ over $\f$, is also hard in general. This problem is usually called the value set problem over finite fields, which aims to determine the cardinality of the image set of a polynomial map and to  understand its algebraic or combinatorial structure.
It has a wide variety of applications in number theory, algebraic geometry, coding theory and cryptography.
For details of this problem, please refer to \cite{LN}. Note that if  we denote  \begin{align*}
V(f)=\{ f(x), x\in \f\},
 \end{align*}
 then clearly $S=V(f, \omega)$.

Many interesting bounds were established, and a trivial one gives $$\lceil \frac q d \rceil \leq |V(f)|\leq q.$$
When the upper bound is achieved, $f$ is called a permutation polynomial.
 The theory of permutation polynomials was extensively studied and has many applications in coding theory and cryptography.  On the other hand, when the lower bound is achieved, $f$ is called  having the minimal value set property.

In this paper we are certainly interested in the minimal value set property. Among them,  the class of  monomials, or generally,
 the composition of a permutation polynomial and  a monomial, are the simplest but most important classes.
 We then first state a useful lemma on the irreducibility of the composition of an irreducible polynomial  and  a monomial.

 \begin{lem}\cite{MB, MP} \label{Lemma2.4}
 Suppose $f(x)$ is an irreducible polynomial of degree $n$ over $\f$ and a root of $f(x)$ has order $e$.
 If $t$ satisfies the following conditions, then $f(x^t)$ is also irreducible.

 1. $(t, \frac{q^n-1}e)=1;$

 2. Each prime factor of $t$ divides $e$;

 3. $q^n\equiv 1(\bmod \ 4)$ if $t\equiv0 (\bmod\ 4)$.

 \end{lem}

By this lemma, it is an interesting problem to classify polynomials are both permutating and irreducible.
 For simplicity, in this paper we focus on the simplest classes --linear polynomials, which are clearly both permutation polynomials and irreducible polynomials.

 Composing  the monomials into linear polynomials, it then suffices to consider the irreducible binomials.  Fortunately, the classification of irreducible binomials were established and it will lead many good  constructions of small desired subsets in finite fields.
%
As a special case of Lemma \ref{Lemma2.4}, we have
\begin{lem}\cite{LN, MP}\label{Lemma2.5}
Let $t\geq 2$ be an integer and $\alpha\in\f$.  Let $e$ be the order of $\alpha$ in $\f$  Then the binomial $x^t-a$ is irreducible if and only if the following  conditions are satisfied:

1. $(t, \frac{q^n-1}e)=1;$

2. Each prime factor of $t$ divides $e$;

2. $q \equiv 1(\bmod \ 4)$ if $t\equiv0(\bmod \ 4)$.

\end{lem}


\begin{prob}
From the viewpoint of elementary number theory, we are interested in the choices of $t$. We may ask how to compute the density of $t$ for given $q, h$ and $e$ satisfying the above conditions.
\end{prob}

Combining Theorem \ref{thm2.2} and Lemma \ref{Lemma2.5} we then have
\begin{thm} Let $h>1$. For any $\alpha\in\fh\backslash\f$, let $e$ be the (multiplicative) order of $\alpha$. Suppose $\alpha-f(x)$ is an irreducible polynomial of degree $n$ over $\fh$.
Suppose $h, t$ satisfy the the following conditions:

 1. $(t, \frac{q^{nh}-1}e)=1;$

 2. Each prime factor of $t$ divides $e$;

 3. $q^{nh}\equiv 1(\bmod \ 4)$ if $t\equiv0 (\bmod\ 4)$;

 4.  $nth\leq \sqrt{q}$.

Then the set  $S=V(\alpha-f(x^t))=\{\alpha-f(x^t), x\in\f\}$  contains a non-d-th power in $\fh$. In particular,  choosing $f(x)=\alpha-x$ we
obtain the set  $S=V(\alpha-x^t)=\{\alpha-x^t), x\in\f\}$  containing a non-d-th power in $\fh$.
 \end{thm}

Since one can always take $t=1$, we have the following result.
\begin{cor}
For any $q, 1 \leq h\leq \lfloor \sqrt{q}\rfloor, d\mid q^h-1$, the set  $S=V(\alpha-x)=\{\alpha-x, x\in\f\}$ contains a non-d-th power in $\fh$.  In particular,  there is a set of cardinality $q$ in $\mathbb{F}_{q^{\lfloor \sqrt{q}\rfloor}}$ containing a non-d-th power.
\end{cor}

We now  give more constructions based on a rough classification of these parameters. Please note that they are not complete and the interested readers may give their own constructions.

{\bf Case 1}: Assume both $q$  and  $\frac{q^h-1}e$ are odd.  Since $h>1$, in this case $q^h-1$ is always divisible by $4$, and thus one can always choose $t=2^k,$ where $1\leq k\leq \lfloor\log_2{\sqrt{q}}\rfloor$ and one checks that $t$ satisfies the first three conditions above. Thus we obtain our main constructions below:

 \begin{cor} Let $h>1$ and $d\mid q^h-1$. If both $q$  and  $\frac{q^h-1}e$ are odd, then for any $t=2^k,$ any integer $h$ satisfying $h\leq \sqrt{q}/2^k$, the set  $S=V(\alpha-x^{2^k})=\{\alpha-x^{2^k}, x\in\f\}$    contains a non-d-th power in $\fh$. Let $h=\lfloor q^{\delta}\rfloor$, then there is a set of cardinality $O(q^{\frac 12 +\delta_t})$ in $\mathbb{F}_{q^{\lfloor q^\delta \rfloor}}$ containing a non-d-th power. Here $0<\delta_t<1$ is determined by $(t, q-1)$.
   \end{cor}

 For interested readers, we list one more specified subcase.
   Choosing $h=2$ and $t=2^{\lfloor\log_2{\sqrt{q}}\rfloor-1}\approx \sqrt{q}/2$, we have
    \begin{cor} Let $h=2$ and $d\mid q^2-1$. If both $q$ and  $\frac{q^2-1}e$ are odd, then
 there is a set of cardinality $O(\frac{q}{M})$ in $\ft$ containing a non-square, where
   $$M=\max_{p \mid q-1}\max_{ 2\leq t\leq \lfloor\log_p{q}\rfloor/2}\gcd(p^t, q-1).$$
  \end{cor}

If we denote $v_p(x)$  to be the p-adic oder of $x$, then
 $M=\max_{p \mid q-1} p^{\min\{v_p(q-1),  \lfloor\log_p{q}\rfloor/2\}}$.
 Thus the construction is pretty good if there is a prime divisor $p$ of $q-1$ such that $v_p(q-1)$ is large.
It would be very interesting to construct smaller subsets containing a non  square in $\ft$.

  Note that the above constructions hold when  $\alpha$ is primitive.
 

{\bf Case 2}: Assume $q$ is even and $3\mid q^h-1$.  $(3, \frac{q^h-1}e)=1$.
              In this case we can always choose $t=3^k,$ where $1\leq k\leq \lfloor\log_3{\sqrt{q}}\rfloor$ and one checks that $t$ satisfies the first three conditions above. Thus we obtain another construction below:

\begin{cor} Let $q$ be even, $h>1$ and $d\mid q^h-1$. If $(3, \frac{q^h-1}e)=1$ and $3\mid q-1$,  then for any $t=3^k,$ any integer $h$ satisfying $h\leq \sqrt{q}/3^k$, the set  $S=V(\alpha-x^{3^k})=\{ \alpha-x^{3^k}, x\in\f\}$  contains a non-d-th power in $\fh$.
Let $h=\lfloor q^{\delta}\rfloor$, then there is a set of cardinality $O(q^{\frac 12 +\delta_t})$ in $\mathbb{F}_{q^{\lfloor q^\delta \rfloor}}$ containing a non-d-th power.
  \end{cor}

{\bf Remark\ 1. }  One can also take another prime divisor of $q-1$ other than 3.  In fact, the construction works better if $q-1$ has a small prime divisor $r$ and the $r$-adic order $v_r(q-1)$ is large.
 Since in $\fh$ we have at least $\phi(q^h-1)=(q^h)^{1-\epsilon}$ primitive elements, we obtain the
 following general result.

\begin{cor} 
Let $h=\lfloor q^{\delta}\rfloor>1$ and $d\mid q^h-1$.
   Let $r$ be a prime divisor of $q-1$ such that the largest prime power part of $q-1$ has  the form $r^s$.
   Then there is a constant $0<\epsilon<1$ such that for a ratio at least $ {q^{-\epsilon h}}$ of $\alpha\in \mathbb{F}_{q^{h}} \backslash\mathbb{F}_{q}$,  the set $S=\{ \alpha-x^t, x\in\mathbb{F}_{q}\}$
of cardinality $1+\frac {q-1} {M(h)}$
  contains a non-d-th  power in $\mathbb{F}_{q^{\lfloor q^\delta\rfloor}}$, where $t$ is the largest power of $r$ such that $t<\sqrt{q}/h$ and $M(h)$ is defined as
   $$M(h)=\max_{r \mid (q-1)} r^{\min\{v_r(q-1),  \lfloor\log_r{q}/2-\log_r h\rfloor\}}.$$
    Here $r$ runs thourgh prime divisors and $v_r(x)$ is the $r$-adic oder of $x$.
 \end{cor}

{\bf Remark\ 2. }  Based on the constructions, we are particularly interested in  the classification of polynomials other than linear polynomials and binomials which are both permutational and irreducible. One may also consider the general interested irreducible polynomials, whose corresponding value set problem has nice estimate.


\section{On small subsets containing a primitive element}

\begin{defn} Let $r>1$ be a positive integer and suppose  $r\mid q-1$.
 We call  $\alpha\in \mathbf{F}_{q}^{*}$ an $r$-free element if
$\gcd(r,(q-1)/ord_{q}(\alpha))=1$.
\end{defn}

\begin{lem}[\cite{Co}] Let $\alpha \in \mathbf{F}_{q}^{*}$, and $r>1$ be
a positive integer such that $r\mid q-1$. Then
$$\sum\limits_{d \mid r}\frac{\mu(d)} {\phi(d)}
\sum\limits_{ord(\chi)=d}\chi(\alpha)
=\left\{\begin{array}{ll}\frac{r}{\phi(r)}, &\alpha\ \mbox{is}\ r\mbox{-free},\\
0, & \mbox{otherwise}.
\end{array}
\right.
$$
where $\phi$ is the Euler totient function, $\mu$ is the
M$\ddot{o}$bius function and $ord(\chi)$ is the order of the
multiplicative character $\chi$ of $\mathbf{F}_{q}$.
\end{lem}

 Let $\alpha\in \mathbf{F}_{q^n}$ and $d$ be a positive
integer. Define
$$
P(d,\alpha)=\frac{\mu(d)}{\phi(d)}\sum\limits_{\chi_d}\chi_d(\alpha),\quad
P(\alpha)=\frac{\phi(q^n-1)}{q^n-1}\sum\limits_{d\mid
q^n-1}P(d,\alpha),
$$
where $\sum\limits_{\chi_d}$ ranges over all multiplicative
characters of $\mathbf{F}_{q^n}$ of order $d$. Thus  (2.1) implies
$$
P(\alpha)=\left\{\begin{array}{ll} 1, &\alpha\ \mbox{is a primitive element},\\
0,&\mbox{otherwise}.
\end{array}
\right.
$$

\begin{thm}
Let $\tau(x)$ denote the number of divisors of $x$. If $\tau(q^n-1)<\frac{\sqrt{q}}{n-1}+1$, then the set $$S=V(\alpha-x)=\{ \alpha-x, x\in\f\}$$
of cardinality $q$ in $\mathbb{F}_{q^n}$ contains a primitive element.
\end{thm}

\begin{proof}
Let $N$ be the number of primitive elements in $\alpha+\f$. Then
 \begin{align*}
 N&=\sum_{a\in\f}\frac{\phi(q^n-1)}{q^n-1}\sum\limits_{d\mid
q^n-1}P(d,\alpha-a)\\
&=\sum_{a\in\f}\frac{\phi(q^n-1)}{q^n-1}\sum\limits_{d\mid
q^n-1} \frac{\mu(d)}{\phi(d)}\sum\limits_{ord(\chi)=d}\chi(\alpha-a)\\
&=\frac{\phi(q^n-1)}{q^n-1}\sum\limits_{d\mid q^n-1} \frac{\mu(d)}{\phi(d)}\sum\limits_{ord(\chi)=d}\sum_{a\in\f}\chi(\alpha-a)\\
&\geq \frac{\phi(q^n-1)}{q^n-1}\ \left ( q -\sum\limits_{d\mid q^n-1, d>1} \frac{1}{\phi(d)}\sum\limits_{ord(\chi)=d} (n-1) \sqrt{q}\right)\\
&= \frac{\phi(q^n-1)}{q^n-1}\ \left ( q-\sum\limits_{d\mid q^n-1, d>1}  (n-1) \sqrt{q}\right)\\
&= \frac{\phi(q^n-1)}{q^n-1}\ \left ( q- (\tau(q^n-1)-1)(n-1) \sqrt{q}\right)
  \end{align*}
\end{proof}

Similarly by applying Theorem \ref{main} we have
\begin{thm}
Let $\tau(x)$ denote the number of divisors of $x$.
Suppose $t$ satisfies the following conditions

 1. $(t, \frac{q^{n}-1}e)=1;$

 2. Each prime factor of $t$ divides $e$;

 3. $q^{n}\equiv 1(\bmod \ 4)$ if $t\equiv0 (\bmod\ 4)$;

 4.  $nt\leq \sqrt{q}$.

If $\tau(q^n-1)<\frac{\sqrt{q}}{nt-1}+1$, then the set $$S=V(\alpha-x^t)=\{ \alpha-x^t, x\in\f\}$$
of cardinality $O(\frac {q} {\gcd{(t, q-1)}})$ in $\mathbb{F}_{q^n}$ contains a primitive element.
\end{thm}

%
%
%


{\bf Acknowledgements.} We thank Professor Daqing Wan for his lectures  on incomplete character sum and Dr. Robert Coulter for his help comments.




%
%
%
%
%
%
%
%
%
%

\end{document}